\documentclass[12pt]{amsart}

\usepackage[colorlinks,
linkcolor=blue,
anchorcolor=blue,
citecolor=blue]{hyperref}
\usepackage{epsfig}
\usepackage{graphicx}
\usepackage{amssymb, amstext, amscd, amsmath}
\usepackage{amsthm, mathrsfs, amsfonts,dsfont}
\usepackage{fullpage}
\usepackage{txfonts}
\usepackage{fancybox}
\usepackage{color}
\usepackage{cite}
\usepackage{comment}
\usepackage{tikz-cd} 

\allowdisplaybreaks

\newtheorem{theorem}{Theorem}[section]
\newtheorem{lemma}{Lemma}[section]
\newtheorem{proposition}{Proposition}[section]
\newtheorem{corollary}{Corollary}[section]

\theoremstyle{definition}
\newtheorem{definition}{Definition}[section]

\newtheorem{remark}{Remark}[section]

\numberwithin{equation}{section}

\begin{document}
\title {Numerical Invariants and Parameter Recovery for Quasi-Homogeneous Submodules $[z^k-w^\ell]$}
\author{Yin Liu}
\address{School of Mathematics and Statistics, Nanyang Normal University,
Nanyang, Henan, 473061, P. R. China} 
\email{lylight@mail.bnu.edu.cn}

\author{Yufeng Lu}
\address{School of Mathematical Sciences, Dalian University of Technology,
Dalian, Liaoning, 116024, P. R. China}
\email{lyfdlut@dlut.edu.cn}

\author{Yixin Yang*}
\address{School of Mathematical Sciences, Dalian University of Technology,
Dalian, Liaoning, 116024, P. R. China}
\email{yangyixin@dlut.edu.cn}

\thanks{*Corresponding author.}

\begin{abstract}
We study Yang's higher numerical invariants for the quasi-homogeneous
submodules
\[
M_{k,\ell}=[z^k-w^\ell]\subset H^2(\mathbb D^2),
\qquad
k,\ell\in\mathbb N,\quad k\neq\ell.
\]
The Hilbert--Schmidt property and the low-order invariants
$\Sigma_0$ and $\Sigma_1$ for this family follow from earlier results
on $M_{\theta,\varphi}$-type submodules.  Our aim is to determine the
complete higher-order sequence.  Using the weighted homogeneous
decomposition associated with $z^k-w^\ell$, we construct explicit
orthonormal bases for the two defect spaces and determine the exact
support of the shifted defect inner products.  For $j\geq1$, set
\[
A_j=\left\lceil\frac{j}{\ell}\right\rceil,
\qquad
B_j=\left\lceil\frac{j}{k}\right\rceil.
\]
We obtain
\[
\Sigma_0(M_{k,\ell})=\frac{\pi^2}{6},
\qquad
\Sigma_j(M_{k,\ell})=F(A_j,B_j),
\]
where $F$ is an explicitly evaluated symmetric function on
$\mathbb N^2$. The sequence is nonincreasing,
tends to zero, and satisfies $\Sigma_j>\Sigma_{j+1}
\Longleftrightarrow
k\mid j\ \text{or}\ \ell\mid j$.
Hence its strict descent set is $k\mathbb N\cup\ell\mathbb N$.
This descent set determines the additive semigroup
$\langle k,\ell\rangle$, while the complete numerical invariant
sequence determines the unordered pair $\{k,\ell\}$.  The order of
the two exponents is not recoverable because of coordinate-swap
symmetry.
\end{abstract}

\maketitle

\section{Introduction}
\label{sec:introduction}

Numerical invariants associated with submodules of the Hardy module
$H^2(\mathbb D^2)$ provide a quantitative way to record the interaction
between the two coordinate shifts.  Their operator-theoretic origin is
closely related to the restriction pair on a submodule, the corresponding
defect spaces, and the core operator; see
\cite{Yang2001,GuoYang2004,Yang2005,Yang2019}.
Although the low-order invariants \(\Sigma_0\) and \(\Sigma_1\) have been studied for several broad classes of Hilbert--Schmidt submodules, explicit formulas for the higher invariants \(\Sigma_j\), \(j\ge2\), and hence for the complete invariant sequence, are known only in a few structured situations.  In particular, the homogeneous setting has played a central
role because the usual degree decomposition reduces the relevant
operator matrices and inner products to finite-dimensional problems. Among the settings in which the numerical invariants can be computed
explicitly, homogeneous principal submodules provide the most developed
class and serve as the natural starting point for the present work.


For a homogeneous principal submodule $[p]$, Azari Key, Lu and Yang
developed a systematic method for studying numerical invariants through
the Gram matrices of the homogeneous defect spaces and the associated
Toeplitz determinants \cite{AzariLuYang2017}.  More recently, several
explicit families have revealed substantially different behaviors.
For the repeated-factor submodule $[(z-w)^2]$, explicit formulas for the
complete sequence were obtained and the monotonicity of Yang's
invariants was verified \cite{LiuLuZu2026square}.  This analysis has
subsequently been extended to the entire repeated-factor family $[(z-w)^r], r\geq1$,
for which the numerical invariant sequence is strictly decreasing
\cite{LiuLuZu2026power}.  On the other hand, the family $[z^r-w^r], r\geq2$,
exhibits a completely different phenomenon: the associated Toeplitz
matrices split into residue-class chains and the numerical invariants
satisfy the block-repetition formula $\Sigma_j([z^r-w^r])
=
\Sigma_{\lceil j/r\rceil}([z-w]), j\geq1$;
see \cite{LiuLuZu2026block}.  Thus repeated factors and distinct
homogeneous factors lead to markedly different invariant sequences.

We consider the quasi-homogeneous family
\[
M_{k,\ell}=[z^k-w^\ell],
\qquad
k,\ell\in\mathbb N,\quad k\neq\ell.
\]
Unlike the homogeneous case, $z^k-w^\ell$ is not homogeneous with
respect to the ordinary total degree, so the usual degree
decomposition does not align the two defect spaces.  Accordingly,
the single-index expression arising from homogeneous degree matching
must be replaced by the general double-index formulation of Yang's
numerical invariants.

The polynomial $z^k-w^\ell$, however, retains a weighted homogeneity.
Indeed, the circle action
\[
(U_\theta f)(z,w)
=
f(e^{i\ell\theta}z,e^{ik\theta}w)
\]
satisfies
\[
U_\theta(z^k-w^\ell)
=
e^{ik\ell\theta}(z^k-w^\ell).
\]
This quasi-homogeneous symmetry has appeared previously in the study of
quotient modules.  In particular, Guo, Wang and Zhang considered the
quotient $[z^k-w^\ell]^\perp$ for coprime integers $k,\ell>1$ and used
the above weighted rotation to obtain a one-variable model and weak
trace-class information \cite{GuoWangZhang2012}.  Quasi-homogeneous
quotient modules over the polydisk have also been investigated from the
viewpoint of essential normality; see \cite{GuoWang2012}.  The problem
addressed here is different: we study Yang's complete numerical invariant
sequence on the submodule side and do not impose a coprimality
assumption on $k$ and $\ell$.

The family considered here is also naturally related to a broader line of
work on Hilbert--Schmidt submodules. Such properties have been studied for
several classes of finitely generated and inner-function-related
submodules; see, for example,
\cite{LuoIzuchiYang2018,ZuYangLu2023,ZuLu2025}.
Of particular relevance here are submodules generated by functions of the
form $\psi(z)-\varphi(w)$.
Zou, Yu and Yang studied the submodules $M_{\psi,\varphi}=[\psi(z)-\varphi(w)]$
and computed, among other quantities, the Hilbert--Schmidt norm of the
associated core operator and several compressed commutators
\cite{ZouYuYang2022}. Hilbert--Schmidtness was subsequently studied for
finitely generated submodules containing
$\theta(z)-\varphi(w)$,
where $\theta$ and $\varphi$ are finite Blaschke products
\cite{ZuYangLu2023}. A further extension showed that
$[q(\theta(z),\varphi(w))]$
is Hilbert--Schmidt whenever $q$ is a homogeneous polynomial and
$\theta,\varphi$ are inner functions \cite{LuZuYang2023}. In particular,
taking
\[
q(u,v)=u-v,\qquad
\theta(z)=z^k,\qquad
\varphi(w)=w^\ell
\]
gives
$[q(\theta(z),\varphi(w))]=[z^k-w^\ell]=M_{k,\ell}$,
so the Hilbert--Schmidt property of the submodules considered here is
already covered by this general theory.

More recently, Zu and Lu established transformation formulas for the
low-order numerical invariant functions of
$M_{\theta,\varphi}$-type submodules \cite{ZuLu2025}.  Applied to
$M=[z-w]$, $\theta(z)=z^k$, and $\varphi(w)=w^\ell$, these formulas
recover the low-order invariants of $M_{k,\ell}$.  The new
problem therefore begins at order $j\geq2$.

Although the general $M_{\theta,\varphi}$-type theory provides
Parseval-frame descriptions of the defect spaces, the monomial
quasi-homogeneous setting admits a more rigid weighted decomposition.
Each nonzero weighted defect component is one-dimensional, and its
distinguished monomial chain has a tridiagonal Gram matrix with
diagonal entries $2$ and adjacent entries $-1$.  This yields explicit
orthonormal bases and reduces the higher shifted inner products to a
finite second-difference calculation, allowing us to determine both
their exact support and their values.

We now state the two main results of the paper. For $j\geq1$, set $A_j=\left\lceil\frac{j}{\ell}\right\rceil$,
$B_j=\left\lceil\frac{j}{k}\right\rceil$,
and define
\[
F(\mu,\nu)
=
\sum_{r=0}^{\infty}
\frac{(r+1)^2}
{(r+\mu)(r+\mu+1)(r+\nu)(r+\nu+1)},
\qquad
\mu,\nu\in\mathbb N.
\]
The function $F$ is symmetric in its two variables.

\begin{theorem}[Complete invariant sequence and staircase structure]
\label{thm:intro-complete}
Let
\[
M_{k,\ell}=[z^k-w^\ell],
\qquad
k,\ell\in\mathbb N,
\qquad
k\ne\ell.
\]
Then
\begin{equation*}
 \Sigma_0(M_{k,\ell})
 =
 \frac{\pi^2}{6},
 \qquad
 \Sigma_j(M_{k,\ell})
 =
 F(A_j,B_j),
 \quad j\geq1.
 \end{equation*}
In particular,
$\Sigma_1(M_{k,\ell})=\frac{\pi^2}{6}-1$.
Moreover, the sequence
$\{\Sigma_j(M_{k,\ell})\}_{j\geq0}$
is nonincreasing, tends to zero, and for every $j\geq1$,
\begin{equation*}
\Sigma_j(M_{k,\ell})>\Sigma_{j+1}(M_{k,\ell})
\quad\Longleftrightarrow\quad
k\mid j\ \text{or}\ \ell\mid j.
\end{equation*}
Hence the strict descent set is
\[
D_{k,\ell}
=
\{j\geq1:
\Sigma_j(M_{k,\ell})>\Sigma_{j+1}(M_{k,\ell})\}
=
k\mathbb N\cup\ell\mathbb N.
\]
\end{theorem}

Viewed together with the diagonal result of~\cite{LiuLuZu2026block}, the two-parameter
formula provides a unified description of the family $[z^k-w^\ell]$.
Indeed, in the specialization $k=\ell=K$, one has
\[
A_j=B_j=\left\lceil\frac{j}{K}\right\rceil,
\]
and hence
\[
\Sigma_j([z^K-w^K])
=
F\left(
\left\lceil\frac{j}{K}\right\rceil,
\left\lceil\frac{j}{K}\right\rceil
\right)
=
\Sigma_{\lceil j/K\rceil}([z-w]),
\qquad j\geq1.
\]
Thus the fixed-length block repetition in the homogeneous case
appears as the diagonal specialization of the two-parameter
staircase.  When $k\neq\ell$, the two ceiling parameters separate and
produce a genuinely quasi-homogeneous staircase structure.


The descent pattern in Theorem~\ref{thm:intro-complete} also carries
arithmetic information. The following result describes precisely what
can be recovered from the descent set and from the full numerical
invariant sequence.

\begin{theorem}[Arithmetic and parameter recovery]
\label{thm:intro-recovery}
Let
\[
M_{k,\ell}=[z^k-w^\ell],
\qquad
k,\ell\in\mathbb N,
\qquad
k\ne\ell.
\]
Then the following assertions hold.

\begin{enumerate}
\item
The descent set
$D_{k,\ell}=k\mathbb N\cup\ell\mathbb N
$
uniquely determines the additive semigroup
$\langle k,\ell\rangle=
\{rk+s\ell:r,s\in\mathbb N_0\}$.

\item
If neither $k$ divides $\ell$ nor $\ell$ divides $k$, then
$D_{k,\ell}$ already determines the unordered pair $\{k,\ell\}$.

\item
In all cases, the complete numerical invariant sequence
$\{\Sigma_j(M_{k,\ell})\}_{j\geq0}$
uniquely determines the unordered pair $\{k,\ell\}$.

\item
For every $j\geq0$,
$\Sigma_j([z^k-w^\ell])=
\Sigma_j([z^\ell-w^k])$.
Consequently, the complete numerical invariant sequence cannot
distinguish the ordered pairs $(k,\ell)$ and $(\ell,k)$.
\end{enumerate}
\end{theorem}

Thus the descent pattern determines the semigroup
$\langle k,\ell\rangle$, while the full invariant sequence recovers
the unordered pair $\{k,\ell\}$.  Coordinate-swap symmetry shows that
this is optimal: the ordered pair cannot be recovered.

The paper is organized as follows. Section~\ref{sec2} fixes the basic notation,
records the general double-index form of Yang's numerical invariants,
and develops the weighted decomposition adapted to $z^k-w^\ell$.
Section~\ref{sec3} proves Theorem~\ref{thm:intro-complete}: we construct
explicit orthonormal bases for the two defect spaces, determine the
exact support and values of the shifted defect inner products, derive
the complete invariant formula, and establish the resulting
two-parameter staircase structure. Section~\ref{sec4} proves
Theorem~\ref{thm:intro-recovery} by analyzing the information contained
in the descent set and in the numerical values of the constant blocks.

\section{Preliminaries}
\label{sec2}

We recall the operator-theoretic notation and the weighted
decomposition needed below.  In contrast with the homogeneous case,
the quasi-homogeneous calculation requires the general double-index
form of Yang's numerical invariants.

\subsection{Notation and basic setup}

We write $\mathbb N=\{1,2,3,\ldots\},
\mathbb N_0=\{0,1,2,3,\ldots\}$.
For $x\in\mathbb R$, we denote by $\lfloor x\rfloor$ and
$\lceil x\rceil$ the floor and ceiling of $x$, respectively. We shall use
the elementary identity
$\lfloor N-x\rfloor=N-\lceil x\rceil, N\in\mathbb Z$.

Let
$\mathbb D^2=
\{(z,w)\in\mathbb C^2:|z|<1,\ |w|<1\}$,
and let $H^2(\mathbb D^2)$ be the Hardy space on the bidisk. Every
$f\in H^2(\mathbb D^2)$ has a unique expansion
\[
f(z,w)
=
\sum_{a,b\ge0}\widehat f(a,b)z^aw^b,
\]
where $\widehat f(a,b)$ denotes the Fourier coefficient of $f$
corresponding to the monomial $z^aw^b$. Equivalently,
$\widehat f(a,b)$ is the coefficient of $z^aw^b$ in the standard
monomial expansion of $f$. Moreover,
$\|f\|^2=
\sum_{a,b\ge0}|\widehat f(a,b)|^2$.
Thus
$\{z^aw^b:a,b\in\mathbb N_0\}$
is the standard orthonormal basis of $H^2(\mathbb D^2)$.

We refer to \cite{ChenGuo2003,DouglasPaulsen1989,Rudin1969}
for general background on Hardy spaces over the polydisk and Hilbert
modules.

\subsection{Defect spaces and numerical invariants}

Let $T_z$ and $T_w$ denote multiplication by the coordinate functions.
A closed subspace
$M\subset H^2(\mathbb D^2)$
is called a submodule if
$T_zM\subset M,
T_wM\subset M$.

For $f\in H^2(\mathbb D^2)$, we denote by $[f]=
\overline{
\{gf:g\in\mathbb C[z,w]\}
}^{\,H^2(\mathbb D^2)}$
the principal submodule generated by $f$.  Equivalently, $[f]$ is the
smallest closed submodule of $H^2(\mathbb D^2)$ containing $f$.

For a submodule $M$, the restrictions
$R_z=T_z|_M,
R_w=T_w|_M$
are commuting isometries on $M$. Their defect spaces are
\[\mathcal E_z(M)=
M\ominus zM=
\ker R_z^*\]
and
\[\mathcal E_w(M)=
M\ominus wM=
\ker R_w^*.\]
When there is no danger of confusion, we write simply
$\mathcal E_z=\mathcal E_z(M),
\mathcal E_w=\mathcal E_w(M)$.
The corresponding orthogonal projections are
\[P_z=I_M-R_zR_z^*,\qquad
P_w=I_M-R_wR_w^*.\]
Since $R_z$ and $R_w$ are isometries,
$P_z=[R_z^*,R_z],
P_w=[R_w^*,R_w]$.

Yang's numerical invariants arise from the interaction of the two
defect spaces $\mathcal E_z$ and $\mathcal E_w$; see
\cite{Yang2001,Yang2019}.  For the present quasi-homogeneous problem,
it is important to retain their general double-index form rather than
the single-index expression available after homogeneous degree matching.

Let
$\{\phi_n:n\in I\}$ and $\{\psi_m:m\in J\}
$
be arbitrary orthonormal bases of $\mathcal E_z$ and $\mathcal E_w$,
respectively.  For $j\geq0$, consider
\begin{equation}
\label{eq:general-double-sum}
\Sigma_j(M)
=
\sum_{n\in I}\sum_{m\in J}
\left|
\langle w^j\phi_n,z^j\psi_m\rangle
\right|^2,
\end{equation}
where the value $+\infty$ is allowed at this stage.  For the
submodules considered in this paper, the finiteness of these quantities
will follow from the explicit formulas obtained in Section~\ref{sec3}.

To connect \eqref{eq:general-double-sum} with its operator-theoretic
formulation, define, for $j\geq0$,
\[
X_j
=
P_wR_z^{*j}R_w^j\big|_{\mathcal E_z}
:
\mathcal E_z\longrightarrow \mathcal E_w.
\]
Since $\psi_m\in \mathcal E_w=\operatorname{ran}P_w$, we have
\[
\begin{aligned}
\langle X_j\phi_n,\psi_m\rangle=
\left\langle
P_wR_z^{*j}R_w^j\phi_n,\psi_m
\right\rangle=
\left\langle
R_w^j\phi_n,R_z^j\psi_m
\right\rangle
=
\langle w^j\phi_n,z^j\psi_m\rangle .
\end{aligned}
\]
Hence, by Parseval's identity in $E_w$,
\[
\Sigma_j(M)
=
\sum_{n\in I}\sum_{m\in J}
|\langle X_j\phi_n,\psi_m\rangle|^2
=
\sum_{n\in I}\|X_j\phi_n\|^2.
\]
Thus $\Sigma_j(M)<\infty$ if and only if $X_j$ is a Hilbert--Schmidt operator,
and in that case $\Sigma_j(M)=\|X_j\|_{\mathcal S_2}^2$.
Here $\mathcal S_2$ denotes the Hilbert--Schmidt class. In particular, \eqref{eq:general-double-sum} is independent of the
particular choices of orthonormal bases of $\mathcal E_z$ and $\mathcal E_w$.

For the two lowest orders, the general formulation above reduces to
the familiar defect and cross-commutator expressions. Since $P_z$
is the orthogonal projection onto $\mathcal E_z$, $P_wP_z$ is the extension of
$X_0=P_w|_{\mathcal E_z}$ by zero to $M$, and hence
$\Sigma_0(M)=
\|P_zP_w\|_{\mathcal S_2}^2$.

Likewise,
\[
[R_z^*,R_w]=[R_z^*,R_w]P_z,
\qquad
\operatorname{ran}[R_z^*,R_w]\subset \mathcal E_w,
\]
so that
$X_1=[R_z^*,R_w]\big|_{\mathcal E_z}$
and therefore
$\Sigma_1(M)
=
\|[R_z^*,R_w]\|_{\mathcal S_2}^2$.

For a principal submodule $[p]$ generated by a homogeneous polynomial
$p$, the nonzero homogeneous components of the two defect spaces are
one-dimensional.  Hence one may choose homogeneous orthonormal bases
$\{\phi_n\}_{n\geq0}\subset\mathcal E_z,
\{\psi_n\}_{n\geq0}\subset\mathcal E_w$,
so that $\phi_n$ and $\psi_n$ have matching total degrees.
Orthogonality of distinct homogeneous components then yields
$\langle w^j\phi_n,z^j\psi_m\rangle=0, \,\,(n\neq m)$,
and therefore
\[
\Sigma_j(M)
=
\sum_{n\geq0}
\left|
\langle w^j\phi_n,z^j\psi_n\rangle
\right|^2.
\]
This is the form used in the homogeneous calculations of
\cite{AzariLuYang2017,LiuLuZu2026square,
LiuLuZu2026block,LiuLuZu2026power}.

For $z^k-w^\ell$ with $k\neq\ell$, such a reduction cannot be assumed
a priori; the weighted grading introduced below will instead determine
the exact pairs $(n,m)$ which contribute to the double sum.

The core operator associated with the commuting pair
$(R_z,R_w)$ is defined by
\[
C_M
=
I_M-R_zR_z^*-R_wR_w^*
+R_zR_wR_z^*R_w^*.
\]
It is a bounded self-adjoint operator on $M$; see
\cite{YangHS2005}.

Following Yang \cite{YangHS2005}, a submodule
$M\subset H^2(\mathbb D^2)$ is called a Hilbert--Schmidt submodule
if its core operator $C_M$ is Hilbert--Schmidt.  By
\cite[Theorem~4.2]{YangHS2005}, this is equivalent to
$\Sigma_0(M)<\infty$
and
$\Sigma_1(M)<\infty$.

\begin{remark}
For a Hilbert--Schmidt submodule, Yang's identity
$\Sigma_0(M)-\Sigma_1(M)=1$ follows from \cite[Theorem~4.2(i)]{YangHS2005}.
We shall not use this identity to determine the low-order invariants
of $M_{k,\ell}$; instead, $\Sigma_0$ and $\Sigma_1$ will both be
recovered from the explicit quasi-homogeneous calculations in
Section~3.  Yang's identity will therefore serve only as a
consistency check.
\end{remark}

\subsection{The weighted decomposition associated with \texorpdfstring{\([z^k-w^{\ell}]\)}{[zk-w{ell}]}}
The weighted circle action associated with \(z^k-w^\ell\) has already played an important role in the study of quasi-homogeneous quotient modules; see \cite{GuoWang2012,GuoWangZhang2012}. We record here the corresponding Fourier decomposition in a form adapted to the principal submodule \(M_{k,\ell}\).

Fix $k,\ell\in\mathbb N$. For $\theta\in\mathbb R$, define
\[
(U_\theta f)(z,w)
=
f(e^{i\ell\theta}z,e^{ik\theta}w).
\]
Since $U_\theta$ only multiplies each monomial by a unimodular scalar,
$\{U_\theta\}_{\theta\in\mathbb R}$ is a strongly continuous unitary
group on $H^2(\mathbb D^2)$.  More precisely, $U_\theta(z^aw^b)=
e^{i(\ell a+kb)\theta}z^aw^b$.

For $q\in\mathbb N_0$, define the weighted homogeneous space
\[
\mathcal H_q^{(k,\ell)}
=
\operatorname{span}
\left\{
z^aw^b:
a,b\in\mathbb N_0,\ 
\ell a+kb=q
\right\},\]
where $\deg_{k,\ell}(z^a w^b):
=\ell a +kb$.

When the pair $(k,\ell)$ is fixed, we simply write $\mathcal H_q$.
We also use the convention $\mathcal H_q=\{0\}, q<0$.

Since the monomials with different bidegree are mutually orthogonal, we have $H^2(\mathbb D^2)
=\bigoplus_{q\geq0}\mathcal H_q,$
where $\bigoplus$ denotes the Hilbert-space orthogonal direct
sum.  Notice that some of the spaces $\mathcal H_q$ may be zero,
especially when $\gcd(k,\ell)>1$.

The orthogonal projection onto $\mathcal H_q$ may equivalently be
written as the Fourier projection
\[
Q_qf
=
\frac1{2\pi}
\int_0^{2\pi}
e^{-iq\theta}U_\theta f\,d\theta,
\qquad f \in H^2(\mathbb D^2), \,\,q \geq 0.\]

Here $Q_q$ is the $q$-th Fourier projection associated with the unitary
representation $\{U_\theta\}_{\theta\in\mathbb R}$. In particular, for a
monomial $z^aw^b$,
\[
Q_q(z^aw^b)
=
\begin{cases}
z^aw^b, & \ell a+kb=q,\\
0, & \ell a+kb\ne q.
\end{cases}
\]
Hence $\operatorname{ran}Q_q=\mathcal H_q$. For convenience, we also set $Q_q=0, q<0$.

Let $p_{k,\ell}(z,w)=
z^k-w^\ell$, the identity
\begin{equation}
\label{eq:Fourier-shift}
Q_s(p_{k,\ell}f)
=
p_{k,\ell}Q_{s-k\ell}f,
\qquad s\in\mathbb Z,
\end{equation}
remains valid for all $s\in\mathbb Z$.
This representation will be useful for passing the weighted
decomposition to invariant subspaces.

The two monomials in $p_{k,\ell}$ have the same weighted degree:
\[
\deg_{k,\ell}(z^k)
=
\ell k,
\qquad
\deg_{k,\ell}(w^\ell)
=
k\ell.
\]
Equivalently,
\begin{equation}
\label{eq:p-weighted-homogeneous}
U_\theta p_{k,\ell}
=
e^{ik\ell\theta}p_{k,\ell}.
\end{equation}
Thus
$p_{k,\ell}\mathcal H_q
\subset
\mathcal H_{q+k\ell}$.

Let $M_{k,\ell}=[p_{k,\ell}]=[z^k-w^\ell]$. For every polynomial $g\in\mathbb C[z,w]$, equation~\eqref{eq:p-weighted-homogeneous} gives
\[
U_\theta(gp_{k,\ell})
=
(U_\theta g)(U_\theta p_{k,\ell})
=
e^{ik\ell\theta}(U_\theta g)p_{k,\ell}.
\]
Since $U_\theta g$ is again a polynomial, it follows that
$U_\theta
\bigl(
\mathbb C[z,w]p_{k,\ell}
\bigr)
\subset
\mathbb C[z,w]p_{k,\ell}$.
By continuity of $U_\theta$ and the definition of the principal
submodule,
$U_\theta M_{k,\ell}\subset M_{k,\ell}$.
Applying the same argument to $U_{-\theta}$ yields the reverse
inclusion, and hence
$U_\theta M_{k,\ell}=M_{k,\ell}$.
Since $M_{k,\ell}$ is closed and invariant under every $U_\theta$,
the Fourier projection
\[
Q_qf
=
\frac1{2\pi}\int_0^{2\pi}e^{-iq\theta}U_\theta f\,d\theta
\]
also belongs to $M_{k,\ell}$ whenever $f\in M_{k,\ell}$.  Thus every
$Q_q$ leaves $M_{k,\ell}$ invariant.

The following decomposition will be used repeatedly in Section~\ref{sec3}.

\begin{lemma}
\label{lem:weighted-module-decomposition}
For every $s\in\mathbb Z$,
\[
M_{k,\ell}\cap\mathcal H_s
=
\begin{cases}
p_{k,\ell}\mathcal H_{s-k\ell},
& s\geq k\ell,
\\[1mm]
\{0\},
& s<k\ell.
\end{cases}
\]
Consequently,
\begin{equation}
\label{eq:weighted-module-direct-sum}
M_{k,\ell}
=
\bigoplus_{q\geq0}
p_{k,\ell}\mathcal H_q.
\end{equation}
\end{lemma}

\begin{proof}
Recall that
\[
M_{k,\ell}
=
\overline{\mathbb C[z,w]\,p_{k,\ell}}^{\,H^2(\mathbb D^2)}
\]
and that $p_{k,\ell}=z^k-w^\ell$ is weighted homogeneous of
weighted degree $k\ell$. Hence
$p_{k,\ell}\mathcal H_r
\subset
\mathcal H_{r+k\ell}, r\geq0$.

We first consider the case $s\geq k\ell$.
Let $h\in\mathcal H_{s-k\ell}$. Since
$\mathcal H_{s-k\ell}\subset\mathbb C[z,w]$, we have
$p_{k,\ell}h
\in
\mathbb C[z,w]\,p_{k,\ell}
\subset
M_{k,\ell}$.
Moreover, by weighted homogeneity,
$p_{k,\ell}h\in\mathcal H_s$.
Therefore
$p_{k,\ell}\mathcal H_{s-k\ell}
\subset
M_{k,\ell}\cap\mathcal H_s$.

Conversely, let $f\in M_{k,\ell}\cap \mathcal H_s$.  Choose polynomials
$\{g_n\}\subset\mathbb C[z,w]$ such that
$p_{k,\ell}g_n\longrightarrow f
\,\,\,\text{in }H^2(\mathbb D^2)$.
Applying $Q_s$ and using \eqref{eq:Fourier-shift} gives
$p_{k,\ell}Q_{s-k\ell}g_n
=
Q_s(p_{k,\ell}g_n)
\longrightarrow f$.

If $s\geq k\ell$, then
$Q_{s-k\ell}g_n\in \mathcal H_{s-k\ell}$, and
$p_{k,\ell}\mathcal H_{s-k\ell}$ is finite-dimensional and hence closed.
Therefore
$f\in p_{k,\ell}\mathcal H_{s-k\ell}$.

If $s<k\ell$, then $Q_{s-k\ell}=0$, so the same convergence gives
$f=0$.  Hence
\[
M_{k,\ell}\cap \mathcal H_s
=
\begin{cases}
p_{k,\ell}\mathcal H_{s-k\ell},&s\geq k\ell,\\
\{0\},&s<k\ell.
\end{cases}
\]

It remains to prove the direct-sum decomposition.
Since each Fourier projection $Q_s$ leaves $M_{k,\ell}$ invariant,
every $f\in M_{k,\ell}$ has the orthogonal expansion
\[
f=\sum_{s\geq0}Q_sf,
\qquad
Q_sf\in M_{k,\ell}\cap \mathcal H_s.
\]
By the formula proved above, $Q_sf=0  (s<k\ell)$, whereas, for $s\geq k\ell$,
$Q_sf\in p_{k,\ell} \mathcal H_{s-k\ell}$.

Hence
\[
f\in
\bigoplus_{s\geq k\ell}
p_{k,\ell} \mathcal H_{s-k\ell}
=
\bigoplus_{q\geq0}
p_{k,\ell} \mathcal H_q,
\]
after the change of index $q=s-k\ell$.  Therefore
$M_{k,\ell}
\subset
\bigoplus_{q\geq0}
p_{k,\ell} \mathcal H_q$.

Conversely, for every $q\geq0$,
$p_{k,\ell} \mathcal H_q
\subset
\mathbb C[z,w]\,p_{k,\ell}
\subset
M_{k,\ell}$.
Since $M_{k,\ell}$ is closed, it follows that
$\bigoplus_{q\geq0}
p_{k,\ell} \mathcal H_q
\subset
M_{k,\ell}$.

Finally, if $q_1\neq q_2$, then
$p_{k,\ell} \mathcal H_{q_j}
\subset
\mathcal H_{q_j+k\ell},
 j=1,2$,
and the two weighted degrees are distinct.  Thus the corresponding
summands are orthogonal.  Consequently,
\[
M_{k,\ell}
=\bigoplus_{q\geq0}
p_{k,\ell} \mathcal H_q.
\]




\end{proof}

Multiplication by $z$ and $w$ shifts the weighted degree by $\ell$
and $k$, respectively.  Hence the orthogonality in \eqref{eq:weighted-module-direct-sum}
is preserved under multiplication by either coordinate function, and
\begin{equation}
\label{eq:zM-weighted}
zM_{k,\ell}
=\bigoplus_{q\geq0}p_{k,\ell}z\mathcal H_q,
\qquad
wM_{k,\ell}
=\bigoplus_{q\geq0}p_{k,\ell}w\mathcal H_q.
\end{equation}
Reindexing the summands and taking orthogonal complements
componentwise gives
\begin{equation}
\label{eq:Ez-weighted}
M_{k,\ell}\ominus zM_{k,\ell}
=
\bigoplus_{q\geq0}
\bigl(p_{k,\ell} \mathcal H_q\ominus p_{k,\ell}z \mathcal H_{q-\ell}\bigr),
\end{equation}
and similarly
\begin{equation}
\label{eq:Ew-weighted}
M_{k,\ell}\ominus wM_{k,\ell}
=\bigoplus_{q\geq0}
\bigl(p_{k,\ell} \mathcal H_q\ominus p_{k,\ell}w \mathcal H_{q-k}\bigr).
\end{equation}

The weighted decomposition
\eqref{eq:weighted-module-direct-sum} is the quasi-homogeneous analogue
of the ordinary homogeneous decomposition used for homogeneous
principal submodules.  The essential difference is that multiplication
by $z$ changes weighted degree by $\ell$, whereas multiplication by
$w$ changes weighted degree by $k$.  It is precisely this asymmetry
that produces two different ceiling parameters
\[\left\lceil\frac{j}{\ell}\right\rceil, \qquad
\left\lceil\frac{j}{k}\right\rceil\]
in the complete numerical invariant formula derived in Section~\ref{sec3}.


\section{Numerical invariants for the submodules \texorpdfstring{\([z^k-w^{\ell}]\)}{[zk-w{ell}]}}
\label{sec3}

Throughout this section, let
\[
M_{k,\ell}
=
[p_{k,\ell}]
=
[z^k-w^\ell]
\subset H^2(\mathbb D^2),
\qquad
k,\ell\in\mathbb N,\quad k\neq\ell.
\]
For brevity, we write $p:=p_{k,\ell}=z^k-w^\ell$ throughout this section whenever no ambiguity can arise.

By the decompositions \eqref{eq:Ez-weighted} and \eqref{eq:Ew-weighted}, the computation of the higher
numerical invariants begins with the nonzero weighted components of
the two defect spaces.  We first construct explicit orthonormal bases
for these components and then use their weighted supports to determine
exactly which pairs contribute to $\langle w^j\varphi_n,z^j\psi_m\rangle,
j\geq1$.

\subsection{Explicit orthonormal bases for the defect spaces}

We begin with a simple dimension observation.

\begin{lemma}
\label{lem:defect-dimension}
For every $q\geq0$,
\[
\dim
\left(
p\mathcal H_q
\ominus
pz\mathcal H_{q-\ell}
\right)
=
\begin{cases}
1,& k\mid q,\\
0,& k\nmid q,
\end{cases}
\]
and
\[
\dim
\left(
p\mathcal H_q
\ominus
pw\mathcal H_{q-k}
\right)
=
\begin{cases}
1,& \ell\mid q,\\
0,& \ell\nmid q.
\end{cases}
\]
Consequently,
\[
M_{k,\ell}\ominus zM_{k,\ell}
=\bigoplus_{n\geq0}
\left(
p\mathcal H_{kn}
\ominus
pz\mathcal H_{kn-\ell}
\right),
\]
and
\[
M_{k,\ell}\ominus wM_{k,\ell}
=\bigoplus_{m\geq0}
\left(
p\mathcal H_{\ell m}
\ominus
pw\mathcal H_{\ell m-k}
\right),
\]
where every summand displayed above is one-dimensional.
\end{lemma}

\begin{proof}
Since multiplication by the nonzero polynomial $p$ is injective,
$\dim p\mathcal H_q=\dim\mathcal H_q$
and
$\dim pz\mathcal H_{q-\ell}=
\dim z\mathcal H_{q-\ell}$. Every monomial $z^a w^b\in\mathcal H_q$ with $a\geq1$ belongs to
$z\mathcal H_{q-\ell}$. Thus the only possible monomial in
$\mathcal H_q$ which does not belong to $z\mathcal H_{q-\ell}$ is a
monomial with $a=0$, namely $w^n$. Such a monomial exists exactly when $q=kn$
for some $n\geq0$. Therefore
\[
\dim\mathcal H_q-\dim z\mathcal H_{q-\ell}
=
\begin{cases}
1,&k\mid q,\\
0,&k\nmid q.
\end{cases}
\]
Since $pz\mathcal H_{q-\ell}$ is a subspace of the
finite-dimensional space $p\mathcal H_q$, the first dimension formula
follows.  Combining it with
\eqref{eq:Ez-weighted}, we obtain
\[
M_{k,\ell}\ominus zM_{k,\ell}
=\bigoplus_{n\geq0}
\left(
p\mathcal H_{kn}
\ominus
pz\mathcal H_{kn-\ell}
\right),
\]
and every summand on the right-hand side is one-dimensional.

The second dimension formula and the corresponding decomposition of
$M_{k,\ell}\ominus wM_{k,\ell}$ follow in the same way after
interchanging $z,k$ with $w,\ell$.
\end{proof}

For $n\geq0$, set
\[
N_n
=
\left\lfloor\frac{n}{\ell}\right\rfloor+1.
\]
Consider the vectors
\[
v_r^{(n)}
=
p z^{kr}w^{n-\ell r},
\qquad
0\leq r\leq N_n-1.
\]
They all belong to the same quasi-homogeneous component
$p\mathcal H_{kn}$.

\begin{lemma}
\label{lem:tridiagonal-chain}
For $0\leq r,s\leq N_n-1$,
\[
\langle v_r^{(n)},v_s^{(n)}\rangle
=
\begin{cases}
2,&r=s,\\
-1,&|r-s|=1,\\
0,&|r-s|\geq2.
\end{cases}
\]
Hence the Gram matrix of
$v_0^{(n)},\ldots,v_{N_n-1}^{(n)}$ is
\[
T_{N_n}
=
\begin{pmatrix}
2&-1&0&\cdots&0\\
-1&2&-1&\ddots&\vdots\\
0&-1&2&\ddots&0\\
\vdots&\ddots&\ddots&\ddots&-1\\
0&\cdots&0&-1&2
\end{pmatrix}.
\]
\end{lemma}


\begin{proof}
For $0\leq r\leq N_n$, set
$\xi_r=
z^{kr}w^{n+\ell-\ell r}$.
Then $\{\xi_r:0\leq r\leq N_n\}$
consists of distinct monomials and hence forms an orthonormal set.
Moreover, 
\[v_r^{(n)}=
(z^k-w^\ell)z^{kr}w^{n-\ell r}=
\xi_{r+1}-\xi_r.\]
Therefore
\[
\begin{aligned}
\langle v_r^{(n)},v_s^{(n)}\rangle=
\langle\xi_{r+1}-\xi_r,\xi_{s+1}-\xi_s\rangle=
2\delta_{r,s}
-\delta_{r+1,s}
-\delta_{r,s+1}.
\end{aligned}
\]
Hence the asserted inner-product relations follow immediately, and the
corresponding Gram matrix is the tridiagonal matrix $T_{N_n}$.
\end{proof}

Before using the tridiagonal Gram matrix to construct the defect
vectors, we verify that the chain considered above is the only one
that contributes to the corresponding defect component. This point is
automatic when $\mathcal H_{kn}$ consists of a single chain, but requires
some care when $\gcd(k,\ell)>1$, since in that case a fixed weighted
homogeneous space may contain several residue chains.

For a statement $E$, let $\mathbf 1_E$ denote its indicator, equal to
$1$ if $E$ holds and to $0$ otherwise. For arbitrary nonnegative
integers $a,b,c,d$, a direct expansion gives
\[
\langle pz^aw^b,pz^cw^d\rangle
=2\mathbf 1_{\{(a,b)=(c,d)\}}-
\mathbf 1_{\{(c,d)=(a+k,b-\ell)\}}-
\mathbf 1_{\{(c,d)=(a-k,b+\ell)\}},\]
where a term involving a negative exponent is understood to be zero.
Consequently,
\begin{equation}
\label{eq:basic-support}
\langle pz^aw^b,pz^cw^d\rangle\neq0
\quad\Longrightarrow\quad
(c-a,d-b)
\in
\{(0,0),(k,-\ell),(-k,\ell)\}.
\end{equation}
In particular, two monomials that interact under the Gram form $(h,g)\longmapsto \langle ph,pg\rangle$
must have $z$-exponents congruent modulo $k$. Hence the monomials in
$\mathcal H_{kn}$ split into mutually orthogonal residue chains modulo
$k$.

Among these chains, the unique one containing the monomial $w^n$ is \[w^n,\;
z^kw^{n-\ell},\;
z^{2k}w^{n-2\ell},\ldots,\] that is, the chain underlying
$v_0^{(n)},\ldots,v_{N_n-1}^{(n)}$. All of its members except $w^n$
are divisible by $z$ and therefore belong to
$z\mathcal H_{kn-\ell}$. Every other residue chain consists entirely
of monomials divisible by $z$, and hence is contained in
$z\mathcal H_{kn-\ell}$.


It follows that only the distinguished chain containing $w^n$
contributes to the one-dimensional defect component
\[
pH_{kn}\ominus pzH_{kn-\ell}.
\]
Lemma~\ref{lem:tridiagonal-chain} therefore reduces the construction of its unit defect
vector to a finite tridiagonal Gram-matrix calculation.

\begin{theorem}
\label{thm:defect-bases}
For $n\geq0$, set $N_n=\left\lfloor\frac{n}{\ell}\right\rfloor+1$
and define
\begin{equation}
\label{eq:phi-basis}
\phi_n
=
\frac{1}{\sqrt{N_n(N_n+1)}}
\sum_{r=0}^{N_n-1}
(N_n-r)\,
p z^{kr}w^{n-\ell r}.
\end{equation}
Then $\{\phi_n:n\geq0\}$
is an orthonormal basis of
$M_{k,\ell}\ominus zM_{k,\ell}$.

Similarly, if
$L_m
=
\left\lfloor\frac{m}{k}\right\rfloor+1$,
and
\begin{equation}
\label{eq:psi-basis}
\psi_m
=
\frac{1}{\sqrt{L_m(L_m+1)}}
\sum_{s=0}^{L_m-1}
(L_m-s)\,
p z^{m-ks}w^{\ell s},
\end{equation}
then $\{\psi_m:m\geq0\}$
is an orthonormal basis of
$M_{k,\ell}\ominus wM_{k,\ell}$.
\end{theorem}

\begin{proof}
We first consider $M_{k,\ell}\ominus zM_{k,\ell}$.
For fixed $n\geq0$, set $v_r^{(n)}=
p z^{kr}w^{n-\ell r},
0\leq r\leq N_n-1$. By the preceding residue-chain decomposition, all residue chains in
$p \mathcal H_{kn}$ other than the distinguished chain containing $w^n$ lie
entirely in $pz \mathcal H_{kn-\ell}$ and are orthogonal to the distinguished
chain.  Moreover, within the distinguished chain,
$v_r^{(n)}\in pz \mathcal H_{kn-\ell},
1\leq r\leq N_n-1$.
Therefore, if $V_n:=
\operatorname{span}
\{v_0^{(n)},\ldots,v_{N_n-1}^{(n)}\}$,
then
\[
p \mathcal H_{kn}\ominus pz \mathcal H_{kn-\ell}
=
V_n\ominus
\operatorname{span}
\{v_1^{(n)},\ldots,v_{N_n-1}^{(n)}\},
\]
where the span on the right is understood to be $\{0\}$ when
$N_n=1$.

Since this defect component is one-dimensional, it suffices to find
a nonzero vector in $V_n$ orthogonal to
$v_1^{(n)},\ldots,v_{N_n-1}^{(n)}$.

Write $N=N_n$ and let
\[
x=\sum_{r=0}^{N-1}c_rv_r^{(n)},
\qquad
c=(c_0,\ldots,c_{N-1})^T.
\]
By Lemma~\ref{lem:tridiagonal-chain}, for $N\geq2$ the orthogonality conditions
\[
\langle x,v_s^{(n)}\rangle=0,
\qquad 1\leq s\leq N-1,
\]
are equivalent to
\[
-c_{s-1}+2c_s-c_{s+1}=0,
\qquad 1\leq s\leq N-2,
\]
together with
\[
-c_{N-2}+2c_{N-1}=0.
\]
The recurrence shows that $c_s$ is affine in $s$, and the terminal
condition then gives
\[
c_s=C(N-s),
\qquad
0\leq s\leq N-1.
\]
Choosing $C=1$, we obtain
\[
c=(N,N-1,\ldots,1)^T.
\]
The formula also covers $N=1$, for which $c=(1)$.  Thus, in all cases,
\[
T_Nc=(N+1)e_0,
\qquad
e_0=(1,0,\ldots,0)^T.
\]



Consequently,
\[
\widetilde\phi_n
:=
\sum_{r=0}^{N-1}(N-r)v_r^{(n)}
\]
is orthogonal to
$v_1^{(n)},\ldots,v_{N-1}^{(n)}$, and hence spans $p\mathcal H_{kn}\ominus pz\mathcal H_{kn-\ell}$.

Moreover,
\[
\|\widetilde\phi_n\|^2
=
c^*T_Nc
=
c^*(N+1)e_0
=
N(N+1).
\]
Thus $\phi_n$ defined in \eqref{eq:phi-basis} is the unit vector
spanning $p \mathcal H_{kn}\ominus pz \mathcal H_{kn-\ell}$.

For $n\neq n'$, the vectors $\phi_n$ and $\phi_{n'}$ belong to the
distinct weighted homogeneous spaces $\mathcal H_{kn+k\ell}$
and $\mathcal H_{kn'+k\ell}$,
respectively, and are therefore orthogonal. Using the orthogonal
decomposition
\[
M_{k,\ell}\ominus zM_{k,\ell}
=\bigoplus_{n\geq0}
\left(
p\mathcal H_{kn}
\ominus
pz\mathcal H_{kn-\ell}
\right),
\]
we conclude that
$\{\phi_n:n\geq0\}$ is an orthonormal basis of
$M_{k,\ell}\ominus zM_{k,\ell}$.

The proof for $M_{k,\ell}\ominus wM_{k,\ell}$ is symmetric.
For fixed $m\geq0$, the distinguished chain in
$p\mathcal H_{\ell m}$ is
\[p z^m,\;
p z^{m-k}w^\ell,\;
\ldots,\;
p z^{m-k(L_m-1)}w^{\ell(L_m-1)},\]
whose Gram matrix is $T_{L_m}$. Applying the same argument with the
coefficient vector
$(L_m,L_m-1,\ldots,1)^{\mathsf T}$
shows that the vector in \eqref{eq:psi-basis} is the unit vector
spanning $p\mathcal H_{\ell m}
\ominus
pw\mathcal H_{\ell m-k}$.
The distinct weighted homogeneous components are mutually orthogonal,
and hence
$\{\psi_m:m\geq0\}$ is an orthonormal basis of
$M_{k,\ell}\ominus wM_{k,\ell}$.
\end{proof}

\begin{remark}
Unlike the general Parseval-frame descriptions for
$M_{\theta,\phi}$-type submodules in \cite{ZuLu2025}, the vectors
in Theorem~\ref{thm:defect-bases} form genuine orthonormal bases supported in individual
weighted homogeneous components.  This localization is what permits
the exact support calculation for the higher shifted defect inner
products.
\end{remark}

\subsection{Defect inner products and their exact support}

Since $M_{k,\ell}$ is not homogeneous in the ordinary sense when
$k\neq\ell$, we use the general double-sum form of Yang's numerical
invariants:
\begin{equation}
\label{eq:general-Sigma}
\Sigma_j(M_{k,\ell})
=
\sum_{n=0}^{\infty}
\sum_{m=0}^{\infty}
\left|
\langle w^j\phi_n,z^j\psi_m\rangle
\right|^2,
\qquad j\geq0.
\end{equation}
The quasi-homogeneous structure will reduce this double sum to a single
series.

Although $\Sigma_0(M_{k,\ell})$ is already known from the
$M_{\theta,\phi}$-type theory, a direct computation is useful here
because it uses the same chain structure as the higher-order case.

\begin{lemma}
\label{lem:sigma0-support}
For $n,m\geq0$, $\langle\phi_n,\psi_m\rangle\neq0$
only if there exists $h\geq0$ such that $n=\ell h, m=kh$.
For such $h$, $\langle\phi_{\ell h},\psi_{kh}\rangle=
\frac1{h+1}$.
\end{lemma}

\begin{proof}
A typical term in $\phi_n$ has the form $pz^{kr}w^{n-\ell r}$,
whereas a typical term in $\psi_m$ has the form $pz^{m-ks}w^{\ell s}$.
By \eqref{eq:basic-support}, if their inner product is nonzero, there is
an $\varepsilon\in\{-1,0,1\}$ such that $m-ks-kr=\varepsilon k$,
and
$\ell s-(n-\ell r)=-\varepsilon\ell$.
Hence
\[
m=k(r+s+\varepsilon),
\qquad
n=\ell(r+s+\varepsilon).
\]

Putting $h=r+s+\varepsilon$
gives $n=\ell h$ and $m=kh$.
Since $n,m\geq0$, necessarily $h\geq0$.

Conversely, suppose $n=\ell h$ and
$m=kh$. To express $\phi_{\ell h}$ and $\psi_{kh}$ with respect to the same chain, set
\[
u_t
=
p z^{kt}w^{\ell(h-t)},
\qquad
0\leq t\leq h.
\]
Since
$N_{\ell h}=L_{kh}=h+1$,
Theorem~\ref{thm:defect-bases} gives
\[
\phi_{\ell h}
=
\frac{1}{\sqrt{(h+1)(h+2)}}
\sum_{t=0}^{h}
(h+1-t)u_t.
\]

On the other hand,
\[
\psi_{kh}
=
\frac{1}{\sqrt{(h+1)(h+2)}}
\sum_{s=0}^{h}
(h+1-s)\,
p z^{k(h-s)}w^{\ell s}.
\]
Making the change of index
$t=h-s$, we have
\[
\psi_{kh}
=
\frac{1}{\sqrt{(h+1)(h+2)}}
\sum_{t=0}^{h}
(t+1)u_t.
\]

Thus, with respect to the common chain
$u_0,\ldots,u_h$, the coefficient vectors of
$\phi_{\ell h}$ and $\psi_{kh}$ are respectively
\[
c
=
(h+1,h,\ldots,2,1)^{\mathsf T},
\qquad
d
=
(1,2,\ldots,h,h+1)^{\mathsf T}.
\]
By Lemma~\ref{lem:tridiagonal-chain}, the Gram matrix of
$u_0,\ldots,u_h$ is $T_{h+1}$.

Then we have $T_{h+1}d=(0,\ldots,0,h+2)^T$.
Therefore $c^TT_{h+1}d=h+2$,
and hence
\[
\langle\phi_{\ell h},\psi_{kh}\rangle
=
\frac{h+2}{(h+1)(h+2)}
=
\frac1{h+1}.
\]
\end{proof}

\begin{theorem}
\label{thm:sigma0}
For $M_{k,\ell}=[z^k-w^\ell]$, we have
$\Sigma_0(M_{k,\ell})=\frac{\pi^2}{6}$.
\end{theorem}

\begin{proof}
By Lemma~\ref{lem:sigma0-support},
\[
\begin{aligned}
\Sigma_0(M_{k,\ell})
&=
\sum_{h=0}^{\infty}
\left|
\langle\phi_{\ell h},\psi_{kh}\rangle
\right|^2=
\sum_{h=0}^{\infty}\frac1{(h+1)^2}
=
\frac{\pi^2}{6}.
\end{aligned}
\]
\end{proof}

We now turn to the higher-order inner products. For $j\geq1$, we
determine both the exact pairs $(n,m)$ for which
$\langle w^j\varphi_n,z^j\psi_m\rangle$
is nonzero and the corresponding inner-product values. Set
\begin{equation}
\label{eq:AjBj}
A_j
=
\left\lceil\frac{j}{\ell}\right\rceil,
\qquad
B_j
=
\left\lceil\frac{j}{k}\right\rceil.
\end{equation}

\begin{theorem}
\label{thm:higher-support}
Let $j\geq1$, and set
\[
A=A_j=\left\lceil\frac{j}{\ell}\right\rceil,
\qquad
B=B_j=\left\lceil\frac{j}{k}\right\rceil.
\]
If
$\langle w^j\phi_n,z^j\psi_m\rangle\neq0, n,m\geq0$,
then there exists a unique integer $h$ such that
\begin{equation}
\label{eq:nm-h}
n=\ell h-j,
\qquad
m=kh-j.
\end{equation}
Necessarily,
$h\geq\max\{A,B\}$.

For every integer $h\geq\max\{A,B\}$, we have
\begin{equation}
\label{eq:exact-alpha}
\left\langle
w^j\phi_{\ell h-j},
z^j\psi_{kh-j}
\right\rangle
=
\begin{cases}
0,
\,\,\,\,\,\,\,\,\,\,\,\,\,\,\,\,\,\,\,\,\,\,\,\,\,\,\,\,\,\,\,\,\,\,\,\,\,\,\,\,\,\,\,\,\,\,\,\,\,\,\,\,\,\,\,\,\,\,\,\,\,\,\,\,\,\,\,\,\,\,\,\,\,\,\,\,\,\,\,\,\,\,\,\,\,\,\max\{A,B\}\leq h<A+B-1,
\\[3mm]
\displaystyle
-\frac{h-A-B+2}
{
\sqrt{
(h-A+1)(h-A+2)
(h-B+1)(h-B+2)
}
}, \,\,\,h\geq A+B-1.
\end{cases}
\end{equation}
Consequently,
\begin{equation}
\label{eq:exact-higher-support}
\langle w^j\phi_n,z^j\psi_m\rangle\neq0
\end{equation}
if and only if there exists an integer $h\geq A+B-1$ such that
$n=\ell h-j,
m=kh-j$.
\end{theorem}

\begin{proof}
Recall from \eqref{eq:phi-basis} and \eqref{eq:psi-basis} that
\[
\phi_n
=
\frac{1}{\sqrt{N_n(N_n+1)}}
\sum_{\rho=0}^{N_n-1}
(N_n-\rho)
p z^{k\rho}w^{n-\ell\rho},
\]
where
$N_n=\left\lfloor\frac{n}{\ell}\right\rfloor+1$,
and
\[
\psi_m
=
\frac{1}{\sqrt{L_m(L_m+1)}}
\sum_{\sigma=0}^{L_m-1}
(L_m-\sigma)
p z^{m-k\sigma}w^{\ell\sigma},
\]
where
$L_m=\left\lfloor\frac{m}{k}\right\rfloor+1$.

Suppose first that
$\langle w^j\phi_n,z^j\psi_m\rangle\neq0$.
Since the expansions of $\phi_n$ and $\psi_m$ above are finite, there must exist indices
\[
0\leq\rho\leq N_n-1,
\qquad
0\leq\sigma\leq L_m-1,
\]
such that
\[
\left\langle
p z^{k\rho}w^{n-\ell\rho+j},
p z^{m-k\sigma+j}w^{\ell\sigma}
\right\rangle
\neq0.
\]
By the basic support relation \eqref{eq:basic-support}, there exists
$\varepsilon\in\{-1,0,1\}$
such that
$m-k\sigma+j-k\rho=\varepsilon k$
and
$\ell\sigma-(n-\ell\rho+j)=-\varepsilon\ell$.

Equivalently, we have
$m+j=k(\rho+\sigma+\varepsilon)$
and
$n+j=\ell(\rho+\sigma+\varepsilon)$.

Define
$h=\rho+\sigma+\varepsilon\in\mathbb Z$.
Then we get
$n=\ell h-j,
m=kh-j$,
which proves \eqref{eq:nm-h}. The integer $h$ is unique, since
\[
h=\frac{n+j}{\ell}=\frac{m+j}{k}.
\]

Since $n,m\geq0$, \eqref{eq:nm-h} implies
$\ell h-j\geq0
$ and $
kh-j\geq0$.
Hence, we have
$h\geq\frac{j}{\ell}
$ and $h\geq\frac{j}{k}$.
As $h$ is an integer,
\[
h\geq
\left\lceil\frac{j}{\ell}\right\rceil=A
\qquad\text{and}\qquad
h\geq
\left\lceil\frac{j}{k}\right\rceil=B.
\]
Therefore $h\geq\max\{A,B\}.$
In particular, all defect vectors appearing below have nonnegative
indices.

We now fix an integer
$h\geq\max\{A,B\}$
and calculate
$\left\langle
w^j\phi_{\ell h-j},
z^j\psi_{kh-j}
\right\rangle$.

For
$n=\ell h-j$, the identity
$\lfloor N-x\rfloor=N-\lceil x\rceil, N\in\mathbb Z$
gives
\[
\begin{aligned}
N_n
&=
\left\lfloor
\frac{\ell h-j}{\ell}
\right\rfloor+1=
\left\lfloor
h-\frac{j}{\ell}
\right\rfloor+1=
h-\left\lceil\frac{j}{\ell}\right\rceil+1=
h-A+1.
\end{aligned}
\]

Similarly, for
$m=kh-j$, we obtain
\[
L_m
=
\left\lfloor
\frac{kh-j}{k}
\right\rfloor+1
=
h-B+1.
\]
Since $h\geq\max\{A,B\}$, both $N_n$ and $L_m$ are positive integers.

For $0\leq r\leq h$, set $
u_r=
p z^{kr}w^{\ell(h-r)}.$
Then the Gram matrix of
$u_0,u_1,\ldots,u_h$
is the tridiagonal matrix $T_{h+1}$ from
Lemma~\ref{lem:tridiagonal-chain}.

Using the formula for $\phi_n$, we obtain
\begin{align}
w^j\phi_{\ell h-j}
&=
\frac{1}
{\sqrt{(h-A+1)(h-A+2)}}
\sum_{r=0}^{h-A}
(h-A+1-r)
p z^{kr}w^{\ell h-j-\ell r+j}
\nonumber\\
&=
\frac{
\displaystyle
\sum_{r=0}^{h-A}
(h-A+1-r)u_r
}
{\sqrt{(h-A+1)(h-A+2)}}.
\label{eq:shifted-phi-1}
\end{align}

Similarly,
\begin{align}
z^j\psi_{kh-j}
&=
\frac{1}
{\sqrt{(h-B+1)(h-B+2)}}
\sum_{\sigma=0}^{h-B}
(h-B+1-\sigma)
p z^{kh-j-k\sigma+j}w^{\ell\sigma}
\nonumber\\
&=
\frac{1}
{\sqrt{(h-B+1)(h-B+2)}}
\sum_{\sigma=0}^{h-B}
(h-B+1-\sigma)
p z^{k(h-\sigma)}w^{\ell\sigma}.
\label{eq:shifted-psi-1}
\end{align}
Making the change of index
$r=h-\sigma$,
we have
\[
\sigma=0,\ldots,h-B
\quad\Longleftrightarrow\quad
r=h,\ldots,B.
\]
Therefore \eqref{eq:shifted-psi-1} becomes
\begin{equation}
\label{eq:shifted-psi}
z^j\psi_{kh-j}
=
\frac{
\displaystyle
\sum_{r=B}^{h}
(r-B+1)u_r
}
{\sqrt{(h-B+1)(h-B+2)}}.
\end{equation}

Let
\[
c=(c_0,c_1,\ldots,c_h)^T,
\qquad
d=(d_0,d_1,\ldots,d_h)^T,
\]
where
\[
c_r
=
\begin{cases}
h-A+1-r,
&0\leq r\leq h-A,
\\
0,
&h-A<r\leq h,
\end{cases}
\]
and
\[
d_r
=
\begin{cases}
0,
&0\leq r<B,
\\
r-B+1,
&B\leq r\leq h.
\end{cases}
\]
Since $j\geq1$, we have
$A\geq1,
B\geq1$.
Hence, we get 
$h-A\leq h-1$,
and therefore
\begin{equation}
\label{eq:ch-zero}
c_h=0.
\end{equation}

Let
$e_0,e_1,\ldots,e_h$
denote the standard coordinate vectors of $\mathbb C^{h+1}$. Since $d_r=0$ for $r<B$ and $d_r=r-B+1$ for $r\geq B$, its
discrete second difference vanishes away from the two boundary
indices.  A direct multiplication, including the boundary indices, gives
\begin{equation}
\label{eq:Td}
T_{h+1}d
=
-e_{B-1}
+
(h-B+2)e_h.
\end{equation}




Since the Gram matrix of $u_0,\ldots,u_h$ is $T_{h+1}$,
equations \eqref{eq:shifted-phi-1} and \eqref{eq:shifted-psi} give
\begin{align}
&
\left\langle
w^j\phi_{\ell h-j},
z^j\psi_{kh-j}
\right\rangle
\nonumber=
\frac{
c^T T_{h+1}d
}
{
\sqrt{
(h-A+1)(h-A+2)
(h-B+1)(h-B+2)
}
}.
\label{eq:inner-product-matrix}
\end{align}

By \eqref{eq:Td} and \eqref{eq:ch-zero},
\[
c^T T_{h+1}d=
-c_{B-1}
+
(h-B+2)c_h=
-c_{B-1}.
\]
It remains only to determine $c_{B-1}$.

If
$\max\{A,B\}\leq h<A+B-1$,
then
$h-A<B-1$,
and hence, by the definition of $c$,
$c_{B-1}=0$.

Therefore,
\[
\left\langle
w^j\phi_{\ell h-j},
z^j\psi_{kh-j}
\right\rangle
=0.
\]

If instead
$h\geq A+B-1$,
then
$B-1\leq h-A$,
so
\[
\begin{aligned}
c_{B-1}
&=
h-A+1-(B-1)=
h-A-B+2.
\end{aligned}
\]
Therefore
\[
\left\langle
w^j\phi_{\ell h-j},
z^j\psi_{kh-j}
\right\rangle
=
-\frac{h-A-B+2}
{
\sqrt{
(h-A+1)(h-A+2)
(h-B+1)(h-B+2)
}
}.
\]
This proves \eqref{eq:exact-alpha}.

Finally, if $h\geq A+B-1$, then
$h-A-B+2\geq1$,
and all four factors in the denominator are strictly positive.
Hence the second quantity in \eqref{eq:exact-alpha} is nonzero.
Combining this fact with the necessity of \eqref{eq:nm-h} proved at
the beginning yields
$\langle w^j\phi_n,z^j\psi_m\rangle\neq0$
if and only if there exists an integer
$h\geq A+B-1$
such that
\[
n=\ell h-j,
\qquad
m=kh-j.
\]
This proves the exact support assertion
\eqref{eq:exact-higher-support}.
\end{proof}

\subsection{The complete invariant sequence and closed forms}

The support formula in Theorem~\ref{thm:higher-support} now yields the complete
higher-order invariant sequence.

\begin{theorem}
\label{thm:complete-sigma}
Let
$k,\ell\in\mathbb N,
k\neq\ell$,
and let
$M_{k,\ell}=[z^k-w^\ell]$. For $j\geq1$, set
\[
A_j=\left\lceil\frac{j}{\ell}\right\rceil,
\qquad
B_j=\left\lceil\frac{j}{k}\right\rceil.
\]
Then
\begin{equation}
\label{eq:complete-sigma}
\Sigma_j(M_{k,\ell})
=
\sum_{r=0}^{\infty}
\frac{(r+1)^2}
{
(r+A_j)(r+A_j+1)
(r+B_j)(r+B_j+1)
},
\qquad j\geq1.
\end{equation}
\end{theorem}

\begin{proof}
By Theorem~\ref{thm:higher-support}, the double sum
\eqref{eq:general-Sigma} has nonzero terms only for
\[
n=\ell h-j,
\qquad
m=kh-j.
\]
Moreover, the corresponding inner product vanishes unless
$h\geq A_j+B_j-1$.

Since
$A_j\geq1,
B_j\geq1$,
every integer
$h\geq A_j+B_j-1$
satisfies
$h\geq A_j
$ and
$h\geq B_j$. Therefore
$\ell h-j\geq0,
kh-j\geq0$,
so the indices of
$\phi_{\ell h-j}$ and
$\psi_{kh-j}$
are admissible throughout the effective summation range.

Hence, writing $A=A_j$ and $B=B_j$,
\[
\Sigma_j(M_{k,\ell})
=
\sum_{h=A+B-1}^{\infty}
\frac{(h-A-B+2)^2}
{
(h-A+1)(h-A+2)
(h-B+1)(h-B+2)
}.
\]
Let
$r=h-A-B+1$.
Then 
\[
\Sigma_j(M_{k,\ell})
=
\sum_{r=0}^{\infty}
\frac{(r+1)^2}
{
(r+A)(r+A+1)
(r+B)(r+B+1)
},
\]
which proves \eqref{eq:complete-sigma}.

Since $A,B\geq1$, for $r\geq1$ we have
\[
0\leq
\frac{(r+1)^2}
{(r+A)(r+A+1)(r+B)(r+B+1)}
\leq
\frac{4}{r^2}.
\]
Hence the series converges absolutely by comparison with
$\sum_{r\geq1}r^{-2}$.
\end{proof}


Define
\[
F(\mu,\nu)
=
\sum_{r=0}^{\infty}
\frac{(r+1)^2}
{(r+\mu)(r+\mu+1)(r+\nu)(r+\nu+1)},
\qquad
\mu,\nu\in\mathbb N.
\]
Then 
\begin{equation}
\label{eq:Sigma-F}
\Sigma_j(M_{k,\ell})
=
F\left(
\left\lceil\frac{j}{\ell}\right\rceil,
\left\lceil\frac{j}{k}\right\rceil
\right),
\qquad
j\geq1.
\end{equation}
Clearly,
$F(\mu,\nu)=F(\nu,\mu)$.

As a first consequence, since
$A_1=B_1=1$,
we obtain
\[
\begin{aligned}
\Sigma_1(M_{k,\ell})
&=
F(1,1)=
\sum_{r=0}^{\infty}
\frac1{(r+2)^2}=
\frac{\pi^2}{6}-1.
\end{aligned}
\]


Thus the direct calculation recovers the previously known low-order
values
\[
\Sigma_0(M_{k,\ell})=\frac{\pi^2}{6},
\qquad
\Sigma_1(M_{k,\ell})=\frac{\pi^2}{6}-1.
\]
These values and the Hilbert--Schmidt property also follow from the
earlier $M_{\psi,\varphi}$- and $M_{\theta,\varphi}$-type theory.
Moreover, by \cite[Theorem~4.2(i)]{YangHS2005},
\[
\|C_{M_{k,\ell}}\|_{S_2}^2
=
\Sigma_0(M_{k,\ell})+\Sigma_1(M_{k,\ell})
=
\frac{\pi^2}{3}-1.
\]
The new content of Theorem~\ref{thm:complete-sigma} is the determination of every higher
invariant $\Sigma_j(M_{k,\ell})$, $j\geq2$.

Formula~\eqref{eq:complete-sigma} already determines the complete invariant sequence in
series form. We next evaluate the function $F(\mu,\nu)$ explicitly,
thereby obtaining closed forms for every member of the sequence.

For the closed-form evaluations below, we use the harmonic numbers
\[
H_n=\sum_{r=1}^n\frac1r,
\qquad
H_n^{(2)}=\sum_{r=1}^n\frac1{r^2},
\qquad n\ge1,
\]
with the convention
$H_0=H_0^{(2)}=0$.





\begin{theorem}
\label{thm:F-closed}
Let $\mu,\nu\in\mathbb N$, and set
\[
u=\min\{\mu,\nu\},
\qquad
\eta=\max\{\mu,\nu\},
\qquad
\tau=\eta-u.
\]
Then
$F(\mu,\nu)=F(u,\eta)$,
and the following formulas hold.

\begin{enumerate}
\item[(i)] If $\tau=0$, equivalently $\mu=\nu=u$, then
\begin{equation}
\label{eq:F-diagonal}
F(u,u)
=
(2u^2-2u+1)
\left(
\frac{\pi^2}{6}-H_{u-1}^{(2)}
\right)
-(2u-1).
\end{equation}

\item[(ii)] If $\tau=1$, then
\begin{equation}
\label{eq:F-adjacent}
F(u,u+1)
=
u-\frac12+\frac{1}{2u}
-u^2
\left(
\frac{\pi^2}{6}-H_u^{(2)}
\right).
\end{equation}

\item[(iii)] If $\tau\geq2$, then
\begin{equation}
\label{eq:F-separated}
F(u,\eta)
=
\frac{
\tau(u+\eta-1)
-
(2u\eta-u-\eta+1)
\bigl(H_{\eta-1}-H_{u-1}\bigr)
}{
\tau(\tau^2-1)
}.
\end{equation}
\end{enumerate}

In particular, if $|\mu-\nu|\geq2$, then
$F(\mu,\nu)\in\mathbb Q$.
\end{theorem}

\begin{proof}
For the diagonal case, a direct partial-fraction decomposition gives
\[
\frac{(r+1)^2}
{(r+u)^2(r+u+1)^2}
=
-\frac{2u(u-1)}{r+u}
+
\frac{(u-1)^2}{(r+u)^2}
+
\frac{2u(u-1)}{r+u+1}
+
\frac{u^2}{(r+u+1)^2}.
\]
Summing over $r\geq0$ and using
$\sum_{r=0}^{\infty}
\left(
\frac{1}{r+u}
-
\frac{1}{r+u+1}
\right)
=
\frac1u$
and
$\sum_{r=0}^{\infty}\frac{1}{(r+u)^2}
=
\frac{\pi^2}{6}-H_{u-1}^{(2)}$
gives \eqref{eq:F-diagonal}.

If $\tau=1$, then $\eta=u+1$, and
\[
\frac{(r+1)^2}
{(r+u)(r+u+1)^2(r+u+2)}
=
\frac{(u-1)^2}{2(r+u)}
+
\frac{2u}{r+u+1}
-
\frac{u^2}{(r+u+1)^2}
-
\frac{(u+1)^2}{2(r+u+2)}.
\]

Summing this identity over $0\leq r\leq N$ and using
\[
\sum_{r=0}^{N}\frac1{r+s}=H_{N+s}-H_{s-1},
\qquad
\sum_{r=0}^{N}\frac1{(r+s)^2}
=
H_{N+s}^{(2)}-H_{s-1}^{(2)},
\]
we note that the coefficients of the first-order harmonic terms sum
to zero:
\[
\frac{(u-1)^2}{2}+2u-\frac{(u+1)^2}{2}=0.
\]
Hence the logarithmic terms cancel as $N\to\infty$, while
$H_{N+u+1}^{(2)}\to\pi^2/6$.  It follows that
\[
F(u,u+1)
=
-\frac{(u-1)^2}{2}H_{u-1}
-2uH_u
+\frac{(u+1)^2}{2}H_{u+1}
-u^2\left(\frac{\pi^2}{6}-H_u^{(2)}\right).
\]
Using
\[
H_{u-1}=H_u-\frac1u,
\qquad
H_{u+1}=H_u+\frac1{u+1},
\]
gives \eqref{eq:F-adjacent}.

Finally, assume
$\tau=\eta-u\geq2$.
We have
\[
\begin{aligned}
\frac{(r+1)^2}
{(r+u)(r+u+1)(r+\eta)(r+\eta+1)}
&=
\frac{(u-1)^2}{\tau(\tau+1)}
\frac{1}{r+u}
-
\frac{u^2}{\tau(\tau-1)}
\frac{1}{r+u+1}\\
&\,\,\,\,\,\,+
\frac{(\eta-1)^2}{\tau(\tau-1)}
\frac{1}{r+\eta}
-
\frac{\eta^2}{\tau(\tau+1)}
\frac{1}{r+\eta+1}.
\end{aligned}
\]

Again, we sum first over $0\leq r\leq N$. Thus
\[
\begin{aligned}
\sum_{r=0}^{N}
\frac{(r+1)^2}
{(r+u)(r+u+1)(r+\eta)(r+\eta+1)}
&=
\frac{(u-1)^2}{\tau(\tau+1)}
\bigl(H_{N+u}-H_{u-1}\bigr)-
\frac{u^2}{\tau(\tau-1)}
\bigl(H_{N+u+1}-H_u\bigr)
\\
&\,\,\,\,\,\,+
\frac{(\eta-1)^2}{\tau(\tau-1)}
\bigl(H_{N+\eta}-H_{\eta-1}\bigr)
-
\frac{\eta^2}{\tau(\tau+1)}
\bigl(H_{N+\eta+1}-H_\eta\bigr).
\end{aligned}
\]


The logarithmic terms cancel because
the four coefficients satisfy
\[
\frac{(u-1)^2}{\tau(\tau+1)}
-\frac{u^2}{\tau(\tau-1)}
+\frac{(\eta-1)^2}{\tau(\tau-1)}
-\frac{\eta^2}{\tau(\tau+1)}
=0.
\]
Letting $N\to\infty$ and simplifying the remaining finite harmonic
terms yields
\[
F(u,\eta)
=
\frac{
\tau(u+\eta-1)
-(2u\eta-u-\eta+1)(H_{\eta-1}-H_{u-1})
}{
\tau(\tau^2-1)
},
\]
which proves \eqref{eq:F-separated}.

Since
$H_{\eta-1}-H_{u-1}$
is a finite rational number, \eqref{eq:F-separated} also shows that
$F(u,\eta)\in\mathbb Q$
whenever $\eta-u\geq2$. Equivalently,
$F(\mu,\nu)\in\mathbb Q$
whenever $|\mu-\nu|\geq2$.
\end{proof}


The diagonal evaluation $F(u,u)$ agrees with the homogeneous formula
in~\cite{LiuLuZu2026block}.  For $k\neq\ell$, both diagonal and off-diagonal pairs
$(A_j,B_j)$ may occur; the latter have no analogue in the symmetric
case.

Theorem~\ref{thm:sigma0}, 
\ref{thm:complete-sigma} and \ref{thm:F-closed} therefore give explicit
closed forms for the entire sequence
$\{\Sigma_j(M_{k,\ell})\}_{j\geq0}$.
We next use the dependence of $\Sigma_j(M_{k,\ell})$ on the two ceiling
parameters
\[
\left\lceil \frac{j}{\ell}\right\rceil
\qquad\text{and}\qquad
\left\lceil \frac{j}{k}\right\rceil
\]
to determine its monotonicity and exact block structure.

\subsection{Monotonicity and the staircase structure}

We next determine exactly when two consecutive numerical invariants are
equal.


\begin{lemma}
\label{lem:F-monotone}
For all $\mu,\nu\in\mathbb N$,
\[
F(\mu+1,\nu)<F(\mu,\nu),
\qquad
F(\mu,\nu+1)<F(\mu,\nu).
\]
More generally, if
\[
\mu',\nu'\in\mathbb N,
\qquad
\mu'\geq\mu,
\qquad
\nu'\geq\nu,
\]
and at least one of the last two inequalities is strict, then
\[
F(\mu',\nu')<F(\mu,\nu).
\]
\end{lemma}

\begin{proof}
For each $r\geq0$,
\[
\begin{aligned}
&
\frac{(r+1)^2}
{(r+\mu)(r+\mu+1)(r+\nu)(r+\nu+1)}
-
\frac{(r+1)^2}
{(r+\mu+1)(r+\mu+2)(r+\nu)(r+\nu+1)}
\\
&=
\frac{2(r+1)^2}
{(r+\mu)(r+\mu+1)(r+\mu+2)(r+\nu)(r+\nu+1)}
>0.
\end{aligned}
\]

Since both series converge absolutely, summation yields
$F(\mu,\nu)>F(\mu+1,\nu)$.
The second inequality follows by symmetry.  Iteration proves the final
assertion.
\end{proof}

For the ceiling functions in \eqref{eq:AjBj}, we have
\[
A_{j+1}-A_j
=
\begin{cases}
1,&\ell\mid j,\\
0,&\ell\nmid j,
\end{cases}
\]
and
\[
B_{j+1}-B_j
=
\begin{cases}
1,&k\mid j,\\
0,&k\nmid j.
\end{cases}
\]

We therefore obtain the complete monotonicity statement.

\begin{theorem}
\label{thm:monotonicity}
Let $k,\ell\in\mathbb N$ with $k\neq\ell$, and let
$M_{k,\ell}=[z^k-w^\ell]$.
Then
\[
\Sigma_0(M_{k,\ell})
>
\Sigma_1(M_{k,\ell})
\geq
\Sigma_2(M_{k,\ell})
\geq
\Sigma_3(M_{k,\ell})
\geq\cdots .
\]
More precisely, for every $j\geq1$,
\begin{equation}
\label{eq:equality-condition}
\Sigma_j(M_{k,\ell})
=
\Sigma_{j+1}(M_{k,\ell})
\iff
k\nmid j
\ \text{and}\
\ell\nmid j,
\end{equation}
whereas
\begin{equation}
\label{eq:strict-condition}
\Sigma_j(M_{k,\ell})
>
\Sigma_{j+1}(M_{k,\ell})
\iff
k\mid j
\ \text{or}\
\ell\mid j.
\end{equation}
Moreover,
\[
\lim_{j\to\infty}\Sigma_j(M_{k,\ell})=0.
\]
\end{theorem}

\begin{proof}
The inequality
$\Sigma_0>\Sigma_1$
follows from
\[
\Sigma_0=\frac{\pi^2}{6},
\qquad
\Sigma_1=\frac{\pi^2}{6}-1.
\]

Let $j\geq1$.  If neither $k$ nor $\ell$ divides $j$, then
\[
A_{j+1}=A_j,
\qquad
B_{j+1}=B_j.
\]
Hence \eqref{eq:Sigma-F} gives
$\Sigma_{j+1}=\Sigma_j$.

If $k\mid j$ or $\ell\mid j$, then at least one of
$A_j,B_j$ increases by one when $j$ is replaced by $j+1$, while the
other one does not decrease.  Lemma~\ref{lem:F-monotone} therefore
gives
$\Sigma_{j+1}<\Sigma_j$.
This proves \eqref{eq:equality-condition} and
\eqref{eq:strict-condition}.

It remains to prove convergence to zero.  Set
$s_j=\min\{A_j,B_j\}$. Then
\[
s_j\longrightarrow\infty
\qquad
(j\to\infty).
\]
Since
$A_j\geq s_j, 
B_j\geq s_j$,
Lemma~\ref{lem:F-monotone} and formula~\eqref{eq:Sigma-F} give
$0\leq
\Sigma_j(M_{k,\ell})
\leq
F(s_j,s_j)$.

Moreover,
\[
\begin{aligned}
F(s_j,s_j)
&=
\sum_{r=0}^{\infty}
\frac{(r+1)^2}
{(r+s_j)^2(r+s_j+1)^2}
\leq
\sum_{r=0}^{\infty}
\frac{1}{(r+s_j)^2}.
\end{aligned}
\]
The last tail tends to zero as $j\to\infty$, because
$s_j\to\infty$.  Hence
$\Sigma_j(M_{k,\ell})\longrightarrow0$.
\end{proof}

Theorem~\ref{thm:monotonicity} shows that the set of strict descent
points is
\[
\bigl\{
j\geq1:
\Sigma_j(M_{k,\ell})>\Sigma_{j+1}(M_{k,\ell})
\bigr\}
=
k\mathbb N\cup\ell\mathbb N.\]



Thus the constant blocks of the sequence are determined exactly by the
increasing enumeration of the set
$k\mathbb N\cup\ell\mathbb N$. More precisely, let
$0=t_0<t_1<t_2<\cdots$
be the increasing enumeration of
$\{0\}\cup k\mathbb N\cup\ell\mathbb N$.
Then
\[
\Sigma_{t_{r-1}+1}
=
\Sigma_{t_{r-1}+2}
=
\cdots
=
\Sigma_{t_r}
>
\Sigma_{t_r+1},
\qquad r\geq1.
\]
In particular, since
$t_1=\min\{k,\ell\}$,
the first block is
\[
\Sigma_1
=
\Sigma_2
=
\cdots
=
\Sigma_{\min\{k,\ell\}}
=
\frac{\pi^2}{6}-1,
\]
followed by a strict decrease.

\begin{remark}
If $k=\ell=K$, then
\[
A_j=B_j=\left\lceil\frac{j}{K}\right\rceil.
\]
Together with $F(s,s)=\Sigma_s([z-w])$, formula~\eqref{eq:Sigma-F} reduces to
\[
\Sigma_j([z^K-w^K])
=
\Sigma_{\lceil j/K\rceil}([z-w]),
\qquad j\geq1,
\]
recovering the block-repetition formula of~\cite{LiuLuZu2026block}.
\end{remark}



\section{Arithmetic consequences and parameter recovery}
\label{sec4}

Throughout this section, let
\[
M_{k,\ell}=[z^k-w^\ell]
\subset H^2(\mathbb D^2),
\qquad
k,\ell\in\mathbb N,\quad k\neq\ell.
\]

Section~\ref{sec3} determines the complete numerical invariant sequence of
$M_{k,\ell}=[z^k-w^\ell]$
through formula~\eqref{eq:Sigma-F}, and Theorem~
\ref{thm:monotonicity} identifies its strict descent points exactly.
We turn to the arithmetic information encoded by the descent set and
the block values.  These data will recover the semigroup
$\langle k,\ell\rangle$ and, ultimately, the unordered pair
$\{k,\ell\}$.

\subsection{The descent set and semigroup recovery}


\begin{definition}
For $M_{k,\ell}=[z^k-w^\ell]$, define its descent set by
\[
\mathcal D_{k,\ell}
=
\{j\in\mathbb N:
\Sigma_j(M_{k,\ell})>\Sigma_{j+1}(M_{k,\ell})\}.
\]
\end{definition}

By Theorem~\ref{thm:monotonicity},
\begin{equation}
\label{eq:recovery-descent}
\mathcal D_{k,\ell}
=
k\mathbb N\cup\ell\mathbb N. 
\end{equation}
Thus the descent set records precisely the union of the positive
multiples of the two exponents.

Set $a=\min\{k,\ell\}$ and $b=\max\{k,\ell\}$.
Then $a<b$ because $k\neq\ell$.

\begin{lemma}
\label{lem:recover-min}
The smaller exponent is determined by the descent set through
\begin{equation}
\label{eq:recover-min}
a
=
\min\mathcal D_{k,\ell}
=
\min
\left\{
j\geq1:
\Sigma_j>\Sigma_{j+1}
\right\}.
\end{equation}
\end{lemma}

\begin{proof}
By \eqref{eq:recovery-descent},
\[
\mathcal D_{k,\ell}
=
a\mathbb N\cup b\mathbb N.
\]
Since $a<b$, the smallest positive element in
$a\mathbb N\cup b\mathbb N$ is $a$.  This proves
\eqref{eq:recover-min}.
\end{proof}


\begin{proposition}
\label{prop:recover-nondivisible}
Suppose that
$a\nmid b$.
Then
\begin{equation}
\label{eq:recover-b-nondivisible}
b
=
\min
\left(
\mathcal D_{k,\ell}\setminus a\mathbb N
\right).
\end{equation}
Consequently, if neither $k$ divides $\ell$ nor $\ell$ divides $k$,
then the descent set $\mathcal D_{k,\ell}$ uniquely determines the
unordered pair $\{k,\ell\}$.
\end{proposition}

\begin{proof}
Since $b\in b\mathbb N\subset\mathcal D_{k,\ell}$ and $a\nmid b$,
we have
$b\in
\mathcal D_{k,\ell}\setminus a\mathbb N$.
Thus this set is nonempty.

Now let
$n\in
\mathcal D_{k,\ell}\setminus a\mathbb N$.
By \eqref{eq:recovery-descent}, we obtain
$n\in a\mathbb N\cup b\mathbb N$.
Since $n\notin a\mathbb N$, necessarily
$n\in b\mathbb N$.
Hence
$n=rb$
for some $r\in\mathbb N$, and therefore
$n\geq b$.
Since $b$ itself belongs to the set, it is its smallest element.
This proves \eqref{eq:recover-b-nondivisible}.

Now suppose that neither $k$ divides $\ell$ nor $\ell$ divides $k$.
By Lemma~\ref{lem:recover-min}, we have 
\[
a=\min\{k,\ell\}
=
\min \mathcal D_{k,\ell}.
\]

Moreover, \eqref{eq:recover-b-nondivisible} gives
$b=
\min
\left(
\mathcal D_{k,\ell}\setminus a\mathbb N
\right)$,
so the descent set also uniquely determines the larger exponent $b$.
Consequently,
$\{k,\ell\}=\{a,b\}$
is uniquely determined by $\mathcal D_{k,\ell}$.
\end{proof}

\begin{remark}
\label{rem:coprime-recovery}
In particular, if
$\gcd(k,\ell)=1
\quad\text{and}\quad
k,\ell>1$,
then neither exponent divides the other.  Hence in this case the
unordered pair $\{k,\ell\}$ is determined entirely by the descent
pattern of the numerical invariant sequence; the actual values of the
$\Sigma_j$ are not needed for parameter recovery.
\end{remark}

For $k,\ell\in\mathbb N$, write
\[
\langle k,\ell\rangle
=
\{rk+s\ell:r,s\in\mathbb N_0\}
\]
for the additive subsemigroup of $\mathbb N_0$ generated by $k$ and
$\ell$. When $\gcd(k,\ell)=1$, this is a numerical semigroup in the
usual sense.

\begin{theorem}
\label{thm:recover-semigroup}
The descent set
\[
\mathcal D_{k,\ell}
=
\{j\geq1:\Sigma_j>\Sigma_{j+1}\}
\]
uniquely determines the additive semigroup
$\langle k,\ell\rangle$.

More precisely, let
$a=\min\mathcal D_{k,\ell}$.
Then the following two mutually exclusive cases occur.

\medskip
\noindent
{\rm (i)}
If
$\mathcal D_{k,\ell}\setminus a\mathbb N\neq\varnothing$,
then, with
$b=
\min
\left(
\mathcal D_{k,\ell}\setminus a\mathbb N
\right)$,
we have $\langle k,\ell\rangle=
\langle a,b\rangle$.

\medskip
\noindent
{\rm (ii)}
If $\mathcal D_{k,\ell}=a\mathbb N$,
then
$\langle k,\ell\rangle=
a\mathbb N_0$.
\end{theorem}

\begin{proof}
Let
\[
a=\min\{k,\ell\},
\qquad
b=\max\{k,\ell\}.
\]
By Lemma~\ref{lem:recover-min},
$a=\min\mathcal D_{k,\ell}$.

{\rm (i)} Assume first that
$\mathcal D_{k,\ell}\setminus a\mathbb N
\neq\varnothing$.
Then $a\nmid b$.  Indeed, if $a\mid b$, then
$b\mathbb N\subset a\mathbb N$ and hence
$\mathcal D_{k,\ell}=
a\mathbb N\cup b\mathbb N=
a\mathbb N$,
contrary to the assumption.  Proposition~\ref{prop:recover-nondivisible}
therefore gives
\[
b
=
\min
\left(
\mathcal D_{k,\ell}\setminus a\mathbb N
\right).
\]
Thus the descent set recovers both $a$ and $b$. Since
$\{k,\ell\}=\{a,b\}$,
and the additive semigroup generated by two integers is independent
of the order of the generators, it follows that
$\langle k,\ell\rangle=\langle a,b\rangle$.

{\rm (ii)} Now suppose that
$\mathcal D_{k,\ell}=a\mathbb N$.
Since $b\in\mathcal D_{k,\ell}$, we must have
$b\in a\mathbb N$.
Hence
$b=qa$
for some integer $q\geq2$.  Therefore
\[
\begin{aligned}
\langle k,\ell\rangle
&=
\langle a,qa\rangle=
\left\{
ra+sqa:
r,s\in\mathbb N_0
\right\}=
\left\{
(r+sq)a:
r,s\in\mathbb N_0
\right\}.
\end{aligned}
\]
Every element of this semigroup is a nonnegative multiple of $a$,
so
$\langle a,qa\rangle\subset a\mathbb N_0$.
Conversely,
\[
na=n\cdot a+0\cdot qa
\in\langle a,qa\rangle
\qquad(n\in\mathbb N_0),
\]
and hence
$\langle a,qa\rangle=a\mathbb N_0$.
This proves the second assertion.
\end{proof}

\begin{remark}
\label{rem:semigroup-vs-generators}
The second case in Theorem~\ref{thm:recover-semigroup} explains why
recovery of the semigroup is strictly easier than recovery of the
particular pair of exponents.  For example,
\[
\langle2,4\rangle
=
\langle2,6\rangle
=
2\mathbb N_0,
\]
and both pairs have the same descent set
$2\mathbb N$.
Thus the descent pattern cannot distinguish the nonminimal generators
$4$ and $6$, although the complete numerical values will distinguish
them.
\end{remark}

\subsection{Recovery of the unordered parameter pair}

It remains to treat the divisibility case, in which the descent set
does not determine the larger exponent.


Assume that $a\mid b$ and write
\[
b=qa,
\qquad
q\in\mathbb N,\quad q\geq2.
\] Since $F$ is symmetric, there is no loss of generality for the
calculation below in assigning the two exponents as
\[
k=a,
\qquad
\ell=qa.
\]
The resulting numerical invariant sequence depends only on the
unordered pair $\{a,qa\}$.

Since
$\mathcal D_{a,qa}=
a\mathbb N$,
the sequence is constant on consecutive blocks of length $a$.  For
$r\geq1$, define the $r$-th block value by
\begin{equation}
\label{eq:Vr-def}
V_r
=
\Sigma_{(r-1)a+1}
=
\Sigma_{(r-1)a+2}
=
\cdots
=
\Sigma_{ra}.
\end{equation}

To express the block values in terms of the function $F$, we only need
the following elementary relation between the block index and the
corresponding ceiling functions.

\begin{lemma}
\label{lem:block-ceiling}
Let $q,a,r\in\mathbb N$, and suppose
$(r-1)a<j\leq ra$.
Then
\begin{equation*}
\label{eq:block-ceiling-1}
\left\lceil\frac{j}{a}\right\rceil=r
\end{equation*}
and
\begin{equation*}
\label{eq:block-ceiling-2}
\left\lceil\frac{j}{qa}\right\rceil
=
\left\lceil\frac{r}{q}\right\rceil.
\end{equation*}
\end{lemma}

\begin{proof}
The first identity follows immediately from
$r-1<\frac{j}{a}\leq r$.

For the second, let
$c=\left\lceil\frac{r}{q}\right\rceil$.
Then $(c-1)q<r\leq cq$, and hence
\[
(c-1)qa<j\leq ra \leq cqa.
\]
Therefore
\[
\left\lceil\frac{j}{qa}\right\rceil=c
=
\left\lceil\frac{r}{q}\right\rceil.
\]
\end{proof}


Combining Lemma~\ref{lem:block-ceiling} with
\eqref{eq:Sigma-F}, we see that both ceiling parameters are
constant throughout the $r$-th block.  Indeed, for
$(r-1)a<j\leq ra$, we have
\[
\left\lceil\frac{j}{a}\right\rceil=r,
\qquad
\left\lceil\frac{j}{qa}\right\rceil
=
\left\lceil\frac{r}{q}\right\rceil.
\]
Hence the common value of the $r$-th block is
\begin{equation}
\label{eq:Vr-formula}
V_r
=
F\left(
r,
\left\lceil\frac{r}{q}\right\rceil
\right),
\qquad r\geq1.
\end{equation}

This formula makes the dependence of the block values on the integer
$q$ explicit and will be used below to recover the larger exponent.

Set
\[
G_r:=F(r,1),
\qquad r\geq1.
\]
By Theorem~\ref{thm:F-closed},
\[
G_r=
\begin{cases}
\dfrac{\pi^2}{6}-1, & r=1,\\[1ex]
2-\dfrac{\pi^2}{6}, & r=2,\\[1ex]
\dfrac{r-1-H_{r-1}}{(r-1)(r-2)}, & r\geq3.
\end{cases}
\]



\begin{theorem}
\label{thm:recover-q}
Suppose
\[
\{k,\ell\}=\{a,qa\},
\qquad q\geq2,
\]
and let $V_r$ be the constant block values defined in
\eqref{eq:Vr-def}.  Then
\begin{equation}
\label{eq:q-recovery}
V_r=G_r
\quad\Longleftrightarrow\quad
1\leq r\leq q.
\end{equation}
Consequently,
\begin{equation*}
\label{eq:q-max-recovery}
q
=
\max
\left\{
r\geq1:
V_r=G_r
\right\}.
\end{equation*}
Equivalently,
\begin{equation}
\label{eq:q-first-deviation}
q+1
=
\min
\left\{
r\geq3:
V_r<G_r
\right\}.
\end{equation}
Thus the complete numerical invariant sequence uniquely determines
the larger exponent
$b=qa$.
\end{theorem}

\begin{proof}
By \eqref{eq:Vr-formula},
\[
V_r
=
F\left(
r,
\left\lceil\frac{r}{q}\right\rceil
\right).
\]

If $1\leq r\leq q$, then
$0<\frac rq\leq1$,
and hence
\[
\left\lceil\frac rq\right\rceil=1.
\]
Therefore
$V_r=F(r,1)=G_r$.

Conversely, if $r>q$, then
$\frac rq>1$,
so
\[
\left\lceil\frac rq\right\rceil\geq2.
\]
By Lemma~\ref{lem:F-monotone}, $F$ is strictly decreasing in its
second variable.  Thus
\[
F\left(
r,
\left\lceil\frac rq\right\rceil
\right)
<
F(r,1),
\]
and consequently
$V_r<G_r$.

This proves \eqref{eq:q-recovery},  and consequently
\[
q=\max\{r\geq1:V_r=G_r\},
\]
while the equivalent formula \eqref{eq:q-first-deviation} follows immediately.

Since $a$ has already been recovered from the first descent by
Lemma~\ref{lem:recover-min}, the identity
$b=qa$
shows that the larger exponent is also uniquely recovered.
\end{proof}


\begin{theorem}
\label{thm:parameter-recovery}
Let
\[
k,\ell,k',\ell'\in\mathbb N,
\qquad
k\neq\ell,
\qquad
k'\neq\ell',
\]
and set
\[
M_{k,\ell}=[z^k-w^\ell],
\qquad
M_{k',\ell'}=[z^{k'}-w^{\ell'}].
\]
If
\begin{equation}
\label{eq:same-sequence}
\Sigma_j(M_{k,\ell})
=
\Sigma_j(M_{k',\ell'})
\qquad
\text{for every }j\geq0,
\end{equation}
then $\{k,\ell\}=
\{k',\ell'\}.$
Thus the complete numerical invariant sequence uniquely determines
the unordered pair of exponents.
\end{theorem}

\begin{proof}
Assume \eqref{eq:same-sequence}.  Since the two sequences agree
term by term, their descent sets agree:
$\mathcal D_{k,\ell}=
\mathcal D_{k',\ell'}$.
Denote the common descent set by $\mathcal D$, and put
$a=\min\mathcal D$.
By Lemma~\ref{lem:recover-min},
\[
a
=
\min\{k,\ell\}
=
\min\{k',\ell'\}.
\]

We distinguish two cases.

\medskip
\noindent
{\it Case 1.}
Suppose
$\mathcal D\setminus a\mathbb N\neq\varnothing$.
Then Proposition~\ref{prop:recover-nondivisible} gives
\[
\max\{k,\ell\}
=
\min
\left(
\mathcal D\setminus a\mathbb N
\right)
=
\max\{k',\ell'\}.
\]
Together with the equality of the smaller exponents, this gives
$\{k,\ell\}=\{k',\ell'\}$.

\medskip
\noindent
{\it Case 2.}
Suppose
$\mathcal D=a\mathbb N$.
Then the larger exponent in each pair is a positive multiple of $a$.
Since the two exponents in each pair are distinct, we may write
\[
\{k,\ell\}
=
\{a,qa\},
\qquad
\{k',\ell'\}
=
\{a,q'a\},
\]
for some integers
$q,q'\geq2$.

Because the complete numerical invariant sequences agree, their
constant block values $V_r$ agree for every $r\geq1$.  Applying
Theorem~\ref{thm:recover-q} to this common sequence gives
\[
q
=
\max\{r:V_r=G_r\}
=
q'.
\]
Hence $qa=q'a$, and therefore
\[
\{k,\ell\}
=
\{a,qa\}
=
\{a,q'a\}
=
\{k',\ell'\}.
\]
This completes the proof.
\end{proof}

\subsection{Coordinate-swap symmetry and the ordered pair}

Although the unordered pair $\{k,\ell\}$ is completely determined by
the invariant sequence, the order of the two exponents is invisible.

\begin{theorem}
\label{thm:ordered-not-recovered}
For all $k,\ell\in\mathbb N$,
\begin{equation}
\label{eq:symmetry-sequence}
\Sigma_j([z^k-w^\ell])
=
\Sigma_j([z^\ell-w^k]),
\qquad j\geq0.
\end{equation}
Consequently, when $k\neq\ell$, the complete numerical invariant
sequence does not determine the ordered pair $(k,\ell)$.
\end{theorem}

\begin{proof}
If $k=\ell$, the assertion is immediate.  We therefore assume
$k\neq\ell$.  For $j=0$, both sides are equal to
$\frac{\pi^2}{6}$.

For $j\geq1$, formula~\eqref{eq:Sigma-F} and the symmetry of $F$ give
\[
\begin{aligned}
\Sigma_j([z^k-w^\ell])
&=
F\left(
\left\lceil\frac j\ell\right\rceil,
\left\lceil\frac jk\right\rceil
\right)=
F\left(
\left\lceil\frac jk\right\rceil,
\left\lceil\frac j\ell\right\rceil
\right)=
\Sigma_j([z^\ell-w^k]).
\end{aligned}
\]
Hence the two complete sequences coincide.

There is also a natural operator-theoretic explanation.  Define
\[
(Uf)(z,w)=f(w,z),
\qquad f\in H^2(\mathbb D^2).
\]
Then $U$ is unitary and
\[
U(z^k-w^\ell)
=
w^k-z^\ell
=
-(z^\ell-w^k).
\]
Since multiplication of a generator by a nonzero scalar does not
change the generated principal submodule,
\[
U[z^k-w^\ell]
=
[z^\ell-w^k].
\]
Moreover, $U$ implements interchange of the two coordinate
multipliers:
\[
UT_z=T_wU,
\qquad
UT_w=T_zU.
\]
Thus the transformation exchanges the ordered coordinate pair
$(T_z,T_w)$ with $(T_w,T_z)$ and simultaneously exchanges the
exponents $k$ and $\ell$.  This coordinate-swap symmetry is consistent
with \eqref{eq:symmetry-sequence}.  Notice, however, that it should
not be confused with unitary equivalence as
$\mathbb C[z,w]$-modules when the ordered coordinate actions are kept
fixed.
\end{proof}

Combining Theorem~\ref{thm:parameter-recovery} and
Theorem~\ref{thm:ordered-not-recovered}, we obtain the precise
parameter-recovery statement:
\[
\{\Sigma_j(M_{k,\ell})\}_{j\geq0}
\quad\text{determines}\quad
\{k,\ell\},
\quad\text{but not the ordered pair }(k,\ell).\]


\begin{corollary}
\label{cor:arithmetic-consequences}
Let
\[
d=\gcd(k,\ell),
\qquad
k=dk_0,
\qquad
\ell=d\ell_0,
\qquad
\gcd(k_0,\ell_0)=1.
\]
Then the descent set determines $d$, while the complete numerical
invariant sequence determines the normalized unordered pair
$\{k_0,\ell_0\}$ and hence the numerical semigroup
$\langle k_0,\ell_0\rangle$.
\end{corollary}

\begin{proof}
Since
$\mathcal D_{k,\ell}=k\mathbb N\cup\ell\mathbb N$,
we have
$d=\gcd\{j:j\in D_{k,\ell}\}$.
Thus $d$ is determined by the descent set. By
Theorem~\ref{thm:parameter-recovery}, the complete invariant sequence
determines $\{k,\ell\}$, and therefore also
\[
\{k_0,\ell_0\}
=
\left\{\frac{k}{d},\frac{\ell}{d}\right\}.
\]
The assertion concerning $\langle k_0,\ell_0\rangle$ follows
immediately.
\end{proof}

When $\gcd(k,\ell)=1$ and $k,\ell>1$, the classical formulas for the
two-generated numerical semigroup $\langle k,\ell\rangle$ give
\[
\operatorname{Frob}\langle k,\ell\rangle=k\ell-k-\ell,
\]
\[
c\bigl(\langle k,\ell\rangle\bigr)=(k-1)(\ell-1),
\]
and
\[
g\bigl(\langle k,\ell\rangle\bigr)
=
\frac{(k-1)(\ell-1)}{2};
\]
see, for example, \cite{RosalesGarciaSanchez2009}. Hence, in the coprime case $k,\ell>1$, the descent set itself---and
therefore also the complete numerical invariant sequence---determines
these classical semigroup invariants.


\medskip

\subsection*{AI Disclosure}
The research questions and the main mathematical ideas of this work originated with the authors. ChatGPT 5.6 Sol (OpenAI) was used as an auxiliary tool during the development of the manuscript to explore and check some intermediate calculations and arguments, and to assist with the organization, exposition, and language editing of the text. All AI-assisted mathematical content was independently reviewed and verified by the authors. The authors take full responsibility for the results, proofs, references, and final text of the article.

\medskip

\subsection*{Acknowledgment} The first author would like to thank Professor Chao Zu for his support and assistance during the preparation of this manuscript.
Y. Liu was supported by the Natural Science Foundation Project in Henan Province (No. 262300421861), the funding program for young backbone teachers in higher education institutions in Henan Province (No. 2024GGJS106), the key research projects of higher education institutions in Henan Province (No. 25B110009) and the general project cultivation fund of Nanyang Normal University (No. 2025PY034). Y. Lu was supported by NSFC (Grant Nos. 12031002 and 12671149). Y. Yang  was supported by NSFC (Grant No.12471117).

%
\subsection*{Conflict of interest}
The authors declare that they have no conflict
of interest. 
\subsection*{Data availability statement}
Data sharing is not applicable to this article as no new data were created or analyzed in this study.

\end{document}